\documentclass[12pt]{amsart}

\usepackage{amsrefs}
\usepackage{latexsym}
\usepackage[all]{xy}

\newcommand{\tot}{\mathsf{T}}

\newcommand{\III}{\mathsf{I}}

\newcommand{\C}{\phc}

\newcommand{\im}{\mathsf{im}}
\newcommand{\ke}{\mathsf{ker}}

\newcommand{\M}{\ph}
\newcommand{\pv}{b}
\newcommand{\ph}{d}
\newcommand{\HH}{\mathsf{H}}

\newcommand{\DDD}{\Omega}
\newcommand{\pvc}{\pv}
\newcommand{\phc}{B}

\newcommand{\xic}{\upsilon}
\newcommand{\xinc}{\xic_n}
\newcommand{\xinnc}{\xic_{n+1}}

\newtheorem{proposition}{Proposition}
\newtheorem{corollary}{Corollary}
\newtheorem{theorem}{Theorem}
\newtheorem{lemma}{Lemma}

\theoremstyle{definition}

\newtheorem{definition}{Definition}
\newtheorem{example}{Example}
\newtheorem{remark}{Remark}

\begin{document}
\title
{Cyclic vs mixed homology}
\author{Ulrich Kr\"ahmer}
\author{Dylan Madden}

\begin{abstract}
The spectral theory of the Karoubi operator 
due to Cuntz and Quillen is
extended to general mixed
(duchain) complexes, that is, chain complexes which are
simultaneously cochain complexes. Connes' coboundary
map $\phc$ can be viewed as a perturbation of the
noncommutative De Rham differential by a polynomial in
the Karoubi operator. The homological impact of such
perturbations is expressed in terms of two short exact
sequences. 

\end{abstract}

\thanks{U.K.~thanks Gabriella B\"ohm, Niels Kowalzig
and Tomasz Maszczyk  
for discussions, and IMPAN Warsaw and
the Wigner Institute Budapest for hospitality. Both
authors thank the referee for a very careful reading of
the manuscript that lead to numerous improvements.}

\maketitle
\tableofcontents 

\section{Introduction and overview}
\subsection{Mixed complexes}
Inspired by Connes' work on cyclic
homology \cite{connes,connesihes}, Dwyer and Kan 
\cite{dk,dk2} initiated 
the study of general 
chain complexes which simultaneously
are cochain complexes: 

\begin{definition}
A \emph{mixed complex} of $R$-modules
is a triple $(\DDD,\pv,\ph)$ where 
$(\DDD,\pv)$ and $(\DDD,\ph)$ are a chain 
respectively a cochain 
complex:
$$
   \xymatrix{\ldots \ar@<1mm>[r]^{\pv} & 
	{\DDD}_{2}
	\ar@<1mm>[r]^{\,\,{\pv}}
	\ar@<1mm>[l]^{\ph} & 
	{\DDD}_{1} 
	\ar@<1mm>[r]^{\!\!{\pv}}
	\ar@<1mm>[l]^{\,\,{\ph}} & 
	{\DDD}_{0}
	\ar@<1mm>[r]^{\,\,\,\,0}
	\ar@<1mm>[l]^{\!\!\ph} & 0
	\ar@<1mm>[l]^{\,\,\,\,0}},\qquad
	\ph^2=\pv^2=0.
$$ 
The \emph{mixed homology} $\HH(\DDD,\pv,\ph)$ is
the homology of  
$(\tot(\DDD),\pv+\ph)$, where
$$
	\tot_n(\DDD):=
	\bigoplus_{i \ge 0}
	\hat\DDD_{n-2i},\qquad
	\hat \DDD_i:= \DDD_{i}/\im\,\xi,\qquad
	\xi := \pv\ph+\ph\pv.
$$
\end{definition}

Dwyer and Kan used the term   
\emph{duchain} rather than  
\emph{mixed} complex, but the latter 
(introduced by Kassel \cite{kassel}) is now the
standard terminology, although it is mostly associated
with the special case $ \xi = 0 $.  

The motivating examples are the 
noncommutative differential forms over an associative
algebra with the De Rham differential $\ph$ and
the Hochschild boundary map $\pv$, 
see Section~\ref{cyclichomology} and 
Example~\ref{exncdiff} below, or \cite[Section~2.6]{loday} for a
detailed account. However, 
mixed complexes appear
in a wide range of contexts,
e.g.~Poisson manifolds \cite{koszul,brylinski},
Lie-Rinehart algebras (Lie algebroids) 
\cite{huebschmann}, and 
Hopf algebras \cite{cm,crainic,hkrs}.

\subsection{The spectral decomposition}\label{asasheaf}
Our aim here is to revisit the construction of cyclic
homology from the perspective of
general mixed complexes.
To this end, we view $\DDD$ as a $k[x]$-module, 
where $k$ is the centre of $R$ and  
$ x $ acts by $\xi$.
Thus ${\DDD}$ defines 
a sheaf of mixed complexes
over the affine line $k$; this generalises the 
spectral decomposition of $\DDD$ considered by Cuntz and Quillen
\cite{cq}. 

The localisation 
$S^{-1} \DDD:=k[x,x^{-1}] \otimes_{k[x]} {\DDD}$ is
contractible as a chain and cochain complex, for 
if $\xi$ is invertible, then we have
$$
	\pv (\xi^{-1} \ph) + 
	(\xi^{-1} \ph) \pv=\mathrm{id},\quad
	\ph(\xi^{-1} \pv) + (\xi^{-1} \pv) 
	\ph=\mathrm{id}.
$$ 
Thus, the only 
stalk of $\DDD$ supporting (co)homology is 
$\hat\DDD=\DDD/\im\,\xi$ 
at $x=0$.
A particularly well-behaved class of mixed complexes
is therefore formed by those which  
are 
globally contractible to $\hat\DDD$:

\begin{definition}\label{skyscrape}
We call $(\DDD,\pv,\ph)$ a 
\emph{(co)homological skyscraper}  if  
$$
	\DDD \rightarrow 
	\hat \DDD=\DDD/\im\,\xi
$$ 
is a quasiisomorphism of 
(co)chain complexes.
\end{definition}

This holds for example when  
${\DDD}=\ke \, \xi \oplus \im \, \xi$ 
so that 
$\im \,\xi \cong S^{-1}{\DDD}$, and 
in particular when $k$ is a field and  
$\xi$ is diagonalisable over $k$. 

\begin{example}\label{appendix0} 
For an example of a non-skyscraper, 
define 
$$ 
	\DDD_n :=\left\{
	\begin{array}{ll}
	R \oplus R \quad & n=0,1,\\
	0 & n>1,
	\end{array}
	\right.\qquad
	\begin{array}{l}
   \ph \colon \DDD_0 \rightarrow \DDD_1,\quad
   (r,s) \mapsto (r,s),\\
   \pv \colon \DDD_1 \rightarrow \DDD_0,\quad
   (u,v) \mapsto (0,u). 
	\end{array}
$$
The homology of $ \DDD$ is $R$ in both degrees and so
is that of $\hat\DDD$, but while the map induced on
homology by the quotient $\DDD \rightarrow \hat\DDD$
is the identity in degree 0 it vanishes in degree 1, so
$\DDD$ is not a homological skyscraper.
\end{example}

We will provide further toy examples that illustrate
the definitions and results throughout the text. 
As a first example of real interest, we mention:

\begin{example}
Consider the De Rham complex $(\DDD,\ph)$ of a compact
oriented Riemannian manifold, and let $\pv$ be the adjoint of
$\ph$ with respect to the Riemannian volume form. Then
$\xi$ is the Laplace operator and the spectral
decomposition of this elliptic (essentially) self-adjoint
operator yields  
$\DDD=\ke \xi \oplus \im \xi$, so $\DDD$ is a
skyscraper and is contractible to $\ke \xi$, the
space of harmonic forms. The results of this paper can
therefore also be viewed
as an abstraction of the Hodge theorem.  
\end{example}

\subsection{Statement of the main results}
The noncommutative
differential forms over an algebra are not a
sky\-scra\-per with respect to the De Rham differential
$d$, see Example~\ref{refquest4}, 
but they are with respect to the coboundary map 
${\phc}$
that defines cyclic homology
(cf.~Section~\ref{cyclichomology} below). Our goal
is to compare cyclic and mixed homology, and we will do
so for more general deformations of $\ph$ by
polynomials in $\xi$:

\begin{definition}
Given any mixed complex $(\DDD,\pv,\ph)$ and 
a sequence of polynomials $c_n \in k[x]$,
we define 
$$
	\phc_n:=c_n\ph_n,\quad
	\xic_n := \pv_{n+1} {\phc}_n + 
	{\phc}_{n-1} \pv_n.
$$
\end{definition}

Our main main result is the following:

\begin{theorem}\label{spektral}
If all $c_n \in k[x]$ 
are invertible in $k [\![x ]\!]$ and
$(\DDD,{\pvc},{\phc})$ is a homological skyscraper,
then for all $n \ge 0$, there are canonical short 
exact sequences
\begin{equation}\label{part1}
	0 \rightarrow \ke\, \pi _n \rightarrow
	\HH_n(\DDD,\pv,\ph) 
	\rightarrow \HH_n(\hat\DDD,\hat\pv,\hat\phc)/
	\ke\,\pi_n 
	\rightarrow 0,
\end{equation}
\begin{equation}\label{part2}
	0 \rightarrow 
	\HH_n(\DDD,\pv,\phc) \rightarrow 
	\HH_n(\hat\DDD,\hat\pv,\hat\phc)
	\rightarrow  
	\HH_{n-1}(\im\,\xi,\pv,\phc)	
	\rightarrow 
	0,
\end{equation}
where $\pi_n$ is the canonical map  
$\HH_n(\hat\DDD,\hat\pv,\hat\phc) \rightarrow 
\HH_n(\hat\DDD/\im \,\hat\pv,0,\hat\phc)$.
\end{theorem}

The maps in (\ref{part2}) are induced by the embedding
$\im\,\xi \rightarrow \DDD$ and the quotient 
$\DDD \rightarrow \hat
\DDD$; those in (\ref{part1}) 
will be described in Section~\ref{fesbew}.

Thus if the two short exact sequences split, then
choosing a split for both yields an isomorphism
$$
	\HH_n(\DDD,\pv,\ph) \cong
	\HH_n(\DDD,\pv,\phc) \oplus 
	\HH_{n-1}(\im\,\xi,\pv,\phc).
$$

Examples~\ref{refquest4}, ~\ref{example1}, 
respectively~\ref{example2} at the end of the paper
illustrate the nontriviality of Theorem~\ref{spektral}
by exhibiting mixed complexes for which 
$\HH(\im\,\xi,\pv,\phc) \neq 0$ respectively 
$\ke\,\pi \neq 0$. 

A key step in the proof is the following computation
that relates the two spectral parameters $ \xi $ and 
$\upsilon $;
as we will explain below, this extends a result of
Cuntz and Quillen.

\begin{proposition}\label{relazion}
We have
\begin{equation}\label{watt}	
	{\xinc}=
	\xi_n c_n -\ph_{n-1} \pv_n f_n
	= \pv_{n+1}\ph_n f_n
	+\xi_n c_{n-1},
\end{equation}
where $f_n:=c_{n}-c_{n-1}$, and 
\begin{equation}\label{brexit}
	({\xinc}-\xi_nc_n)
	({\xinc}-\xi_n c_{n-1})=0.
\end{equation}
\end{proposition}

\subsection{Cyclic homology}\label{cyclichomology}
The most important choice for the polynomials $c_n$
leads to the definition of cyclic homology: 

\begin{definition}\label{heriot}
If  
$$
	c_n=\sum_{i=0}^n (1-x)^i=\frac{1-(1-x)^{n+1}}{x}=
	\sum_{i=0}^n (-1)^i \left(\begin{array}{c}
	\!\!\!n+1\!\!\!
	\\
	\!\!\!i+1\!\!\!
	\end{array}\right)x^i,
$$ 
then 
${\phc}_n=\ph_n \sum_{i=0}^n (\mathrm{id} - \xi_n)^i$
is called the \emph{Connes coboundary map} and
$\HH(\DDD,\pv,\phc)$ the \emph{cyclic
homology} of $\DDD$; furthermore,
$\DDD$ is said to be a \emph{cyclic complex} if  
$\xic=\pv\phc+\phc\pv=0.$
\end{definition}

Theorem~\ref{spektral} relates, in particular, the
mixed homology of a cyclic complex to its cyclic
homology, as long as the constant coefficients 
$n+1$ of $c_n$ are invertible in the ground ring
$k$. If $\xic=0$, we in fact obtain an isomorphism 
$\HH_n(\im\,\xi,\pv,\phc)=\bigoplus_{i \ge 0}
\ke \, \xi_{n-2i} \cap \im\,\xi_{n-2i}$: 

\begin{corollary}
If $(\DDD,\pv,\ph)$ is a cyclic complex of
$ \mathbb{Q} $-vector spaces, 
then there are (noncanonical) isomorphisms of vector spaces 
$$
	\HH_n(\DDD,\pv,\ph) \cong 
	\HH_n(\DDD,\pv,\phc) \oplus \bigoplus_{i \ge 0} 
	\ke\,\xi_{n-1-2i} \cap 
	\im \, \xi_{n-1-2i}.
$$
\end{corollary}

This applies in particular to the
noncommutative differential forms over a unital
associative algebra $A$. Here $(\DDD,\pv)$ is the
normalised Hochschild chain complex of $A$, that is,
\begin{equation}\label{ncdiffforms}
	\DDD_n:=A \otimes_k (A/k)^{\otimes_k n} 
\end{equation}
where $k$ is embedded into $A$ as scalar multiples of
the unit element, and $\pv_n$ is induced by the map
\begin{align*}
 a_0 \otimes _k a_1 \otimes _k \cdots 
	\otimes _k a_n 
& \mapsto  a_0 a_1 \otimes _k a_2 \otimes _k \cdots 
	\otimes _k a_n\\
	& - a_0 \otimes _k a_1a_2 \otimes _k
\cdots 
	\otimes _k a_n + \ldots\\ 
& +(-1)^{n-1}
	a_0 \otimes _k a_1 \otimes _k \cdots 
	\otimes _k a_{n-1} a_n\\
& +(-1)^{n}
	a_na_0 \otimes _k a_1 \otimes _k \cdots 
	\otimes _k a_{n-1}.
\end{align*}
The coboundary map is the noncommutative De Rham
differential $\ph_n$ which is induced by
\begin{equation}\label{extrdeg}
	\ph_n (a_0 \otimes_k \cdots \otimes _k a_n):=
	1 \otimes _k a_0 \otimes _k \cdots \otimes _k
	a_n. 
\end{equation}
This is a cyclic complex, and $\HH(\DDD,\pv,\phc)$ is 
the \emph{cyclic homology} $HC(A)$ of the algebra $A$
\cite{connesihes,loday}.

Considering again a general mixed complex
$(\DDD,\pv,\ph)$, 
the formulas from Proposition~\ref{relazion} reduce 
with $c_n$ as in Definition~\ref{heriot}
to
\begin{equation}\label{cuqu1}
	T_n=(\mathrm{id} -\pv_{n+1} \ph_n) \kappa_n^n,\quad
	\kappa_n^{n+1}=T_n(\mathrm{id}-\ph_{n-1} \pv_n),
\end{equation}
\begin{equation}\label{cuqu2}
	(T_n-\kappa_n^{n+1})(T_n-\kappa_n^n)=0,
\end{equation}
where 
\begin{equation}\label{cuqu3}
	\kappa_n := \mathrm{id}-\xi_n,\qquad
	T_n := \mathrm{id} - \xinc
\end{equation}
are the \emph{Karoubi operators} of the two mixed
complexes $(\DDD,\pv,\ph)$ and $(\DDD,\pv,\phc)$, respectively. 
This generalises \cite[Proposition~3.1]{cq} to
arbitrary mixed complexes and in particular to all
cyclic ones, where $T=\mathrm{id} $ 
(Cuntz and Quillen only considered the
example of noncommutative differential forms).

Our original motivation for the present work was to
extend the results of Cuntz and Quillen and related
work to the so-called twisted cyclic homology
introduced by Kustermans, Murphy and Tuset 
\cite{kmt}, see Example~\ref{exncdiff} for details. 
This and the Hopf-cyclic homology
discovered by
Connes-Moscovici \cite{cm} and Crainic \cite{crainic}
can be viewed as a special case of
Hopf-cyclic homology with coefficients in anti
Yetter-Drinfeld modules, as introduced by Hajac,
Khalkhali, Rangipour and Sommerh\"auser \cite{hkrs}.
Early versions and special instances of the main
results of our paper were key tools in the computation
of the twisted cyclic homology of quantum $SL(2)$ due
to Hadfield and the first author \cite{tomuli}.
Another source of motivation was the work of Shapiro
\cite{shapiro} who investigated the approach via
noncommutative differential forms.
Although the results as formulated
here are fairly technical, we felt it worthwhile to
present them in full generality from the viewpoint of mixed
complexes and hope they will find new applications in
other settings in the future.

\section{Proofs and further material}
\subsection{Quasiisomorphisms}
Before beginning the proofs of the main results, we
remark that what one should call a quasiisomorphism (or
weak equivalence) of mixed complexes is a subtle
question that depends on one's aims 
(see e.g.~\cite[Section~2.5.14]{loday} and
\cite{dk2} for two different choices). 
We will, however, only encounter the simple
case that is covered by the following proposition, which
is a straightforward generalisation of
\cite[Corollary~2.2.3]{loday}:

\begin{proposition}\label{banalitaet}
A morphism   
$ \varphi : (\DDD,\pv,\ph) \rightarrow 
(\DDD',\pv',\ph')$ of mixed
complexes with
$\pv\ph+\ph\pv=\pv'\ph'+\ph'\pv'=0$ 
induces an isomorphism 
on homology if and only if it
induces an isomorphism 
on mixed homology.
\end{proposition}

\begin{example}
Observe that the analogue of the proposition for 
cohomological quasiisomorphisms fails: 
consider for example the two mixed complexes 
$$
	\DDD_n:=
	\left\{
	\begin{array}{ll}
	\mathbb{C} & \quad n=0,\\
	0 & \quad n>0,
	\end{array}
	\right.\qquad
	\DDD'_m:=\mathbb{C},\quad m \in \mathbb{N}  
$$
with $\pv_n=\ph_n=\pv_n'=0$ and 
$$
	\ph'_n:=
	\left\{
	\begin{array}{ll}
	0 & \quad n=2k,\\
	\mathrm{id}  & \quad n=2k+1,
	\end{array}
	\right.\qquad
	k \in \mathbb{N}. 
$$
We have 
$$
	\HH^n(\DDD,0)=\DDD_n \cong 
	\HH^n(\DDD',\ph'),
$$ 
so the map 
$$
	\varphi _n :=
	\left\{
	\begin{array}{ll}
	\mathrm{id}  & \quad n=0,\\
	0  & \quad n>0
	\end{array}
	\right.
$$
is a quasiisomorphism of cochain complexes:
$$
   \xymatrix{
	\ldots \ar@<1mm>[r]^{0} & 
	0
	\ar@<1mm>[r]^{\,\,{0}}
	\ar@<1mm>[l]^{0} \ar[d]^0 & 
	0 
	\ar@<1mm>[r]^{\!\!{0}}
	\ar@<1mm>[l]^{\,\,{0}} 
	\ar[d]^0 & 
	0 
	\ar@<1mm>[r]^{\,\,\,\,0}
	\ar@<1mm>[l]^{\!\! 0} 
	\ar[d]^{0} & 
	0 
	\ar@<1mm>[r]^{\!\!{0}}
	\ar@<1mm>[l]^{\,\,{0}} 
	\ar[d]^0 & 
	\mathbb{C} 
	\ar@<1mm>[r]^{\,\,\,\,0}
	\ar@<1mm>[l]^{\!\! 0} 
	\ar[d]^{\mathrm{id}} & 
	0
	\ar@<1mm>[l]^{\,\,\,\,0}\\
	\ldots \ar@<1mm>[r]^{0} 
	 & 
	\mathbb{C} 
	\ar@<1mm>[r]^{\,\,{0}}
	\ar@<1mm>[l]^{0} & 
	\mathbb{C}  
	\ar@<1mm>[r]^{\!{0}}
	\ar@<1mm>[l]^{\,{\mathrm{id} }} & 
	\mathbb{C} 
	\ar@<1mm>[r]^{\,\,{0}}
	\ar@<1mm>[l]^{0} & 
	\mathbb{C}  
	\ar@<1mm>[r]^{\!{0}}
	\ar@<1mm>[l]^{\,{\mathrm{id} }} & 
	\mathbb{C} 
	\ar@<1mm>[r]^{\,\,\,\,0}
	\ar@<1mm>[l]^{\!\!0} & 0
	\ar@<1mm>[l]^{\,\,\,\,0}
	}
$$ 

However, 
one obtains by direct inspection 
$$
	\HH_n(\DDD',0,\pv') \cong \mathbb{C} ,\qquad
	\HH_n(\DDD,0,0) \cong
	\left\{
	\begin{array}{ll}
	\mathbb{C}  & \quad n=2k,\\
	0  & \quad n=2k+1,
	\end{array}
	\right.\qquad
	k \in \mathbb{N}. 
$$
\end{example}
\begin{remark}
The moral is that, although the r\^oles of $\ph$ and 
$\pv$ are entirely symmetric in $\DDD$, this symmetry is broken  
in the definition of mixed homology, as the
action of $\ph$ is somewhat artificially cut off 
on $\hat\DDD_n \subset \tot_n(\DDD)$.
This changes when one considers the $
\mathbb{Z}_2$-graded periodic homology theories;
however, then there are two variants:  
$$
	\tot_s^{\mathrm{per},\Pi}(\DDD):=
\prod_{j \in \mathbb{N}} 
	\hat\DDD_{s+2j} ,\quad
	\tot_s^{\mathrm{per},\oplus}(\DDD):=
	\bigoplus_{j \in \mathbb{N}} 
	\hat\DDD_{s+2j} ,\quad
s \in
	\mathbb{Z}_2. 
$$
Proposition~\ref{banalitaet} holds in the same way for 
$\HH^{\mathrm{per},\Pi}$, but it is cohomological
quasiisomorphisms rather than homological ones which
induce isomorphisms in $\HH^{\mathrm{per},\oplus}$.
\end{remark}

\subsection{The proof of Proposition~\ref{relazion}}
We now develop the theory that
will lead to a proof of Theorem~\ref{spektral}. The steps are 
illustrated using the example of cyclic
homology, and the first one is the proof of
Proposition~\ref{relazion} in which we relate
the maps $\xi$ and ${\xic}$:

\begin{proof}[Proof of Proposition~\ref{relazion}]
The first equation is obtained by straightforward computation:
\begin{align*}
	{\xinc}
&= {\pvc}_{n+1}{\phc}_n+
	{\phc}_{n-1}{\pvc}_n 
= \pv_{n+1} c_n \ph_n+
	c_{n-1} \ph_{n-1} \pv_n 
	\\
&= \pv_{n+1} \ph_n c_n+
	\ph_{n-1} \pv_n c_{n-1} 
= (\xi_n-\ph_{n-1} \pv_n) c_n
	+\ph_{n-1} \pv_n c_{n-1}
	\\
&= \pv_{n+1}\ph_n c_n
	+(\xi_n-\pv_{n+1}\ph_n) c_{n-1}.
\end{align*}
Thus 
the first factor in (\ref{brexit}) equals $-\ph_{n-1}\pv_n
f_{n+1}$ and the second one $\pv_{n+1}\ph_n f_n$, so
their product equals 0 as $\pv_n\pv_{n+1}=0$. 
\end{proof}
\begin{remark}
If one perturbs not just $\ph_n$ to ${\phc}_n=c_n\ph_n$ 
but also $\pv_n$ to $D_n:=a_n\pv_n$ for some polynomials 
$a_n \in k[x]$, then one has  
$$	
	\phc_{n-1} D_n+ 
	D_{n+1} \phc_n =
	\xi_n a_{n+1}c_n -\ph_{n-1} \pv_n f_n
	= \pv_{n+1}\ph_n f_n
	+\xi_n a_nc_{n-1}
$$
with 
$f_n=a_{n+1}c_n-a_nc_{n-1}$. That is, one obtains
$\xic$ but for $\ph$ perturbed by the
polynomials 
$a_{n+1}c_n$ and in this sense it is sufficient to
focus on deformations of $\ph$ alone.  
\end{remark}

\begin{example}
In the case of cyclic homology
(cf.~Section~\ref{cyclichomology}), we obtain 
$$
	c_n=\frac{1-y^{n+1}}{1-y},\quad
	f_n=y^n,\quad
	y:=1-x.
$$ 
Inserting this into the formulas in Proposition~\ref{relazion} yields the
formulas (\ref{cuqu1})-(\ref{cuqu3}) from
Section~\ref{cyclichomology}.
\end{example}

\subsection{The quasiisomorphism $\DDD \rightarrow \bar\DDD$}
As part of the assumptions of Theorem~\ref{spektral},  
$(\DDD,\pv,\phc)$ is a 
homological skyscraper, so 
$(\im\,{\xic},{\pvc})$ 
has trivial homology. 
We now use this fact to relate the mixed homology of $\DDD$ 
to that of the quotients 
$$
	\bar\DDD:=\DDD/\im\,\xic,\qquad
	\tilde\DDD:=\DDD/(\im\,\xi+\im\,\xic).
$$

In the sequel, $\bar\ph,\bar\pv,\bar\xi$ and
$\tilde\pv,\tilde\ph,\tilde\xi$ refer to 
the structure maps on $\bar\DDD$ respectively 
$\tilde \DDD$.

\begin{lemma}\label{schneideritz}
$(\im\,\xi \cap \im\,\xic,\pv)$ has trivial
homology.
\end{lemma}
\begin{proof}
If $x \in (\im\,\xic \cap \im\,\xi)_n$ and 
$\pv_n x=0$, then as $(\im\,\xic,\pv)$ has no
homology, 
there is $y \in \DDD_{n+1}$ with 
$x=\pv_{n+1} \xinnc y$. By
Proposition~\ref{relazion}, this
equals $\pv_{n+1}\xi_{n+1} c_ny$, so $x \in \pv(\im\,\xic
\cap \im\,\xi)_{n+1}$. 
\end{proof}

\begin{lemma}\label{wlog}
The canonical quotient  
$
	\hat\DDD \rightarrow   
	\tilde\DDD
$
is a quasiisomorphism of chain complexes. In
particular, the quotient $(\DDD,\pv,\ph) \rightarrow 
(\bar\DDD,\bar\pv,\bar\ph)$ induces 
isomorphisms 
$
	\HH(\DDD,\pv,\ph) =
	\HH(\hat\DDD,\hat\pv,\hat\ph) \cong 
	\HH(\bar\DDD,\bar\pv,\bar\ph) =
	\HH(\tilde\DDD,\tilde\pv,\tilde\ph).
$
\end{lemma}
\begin{proof}
We need to show that the kernel 
$$
	\im\,\xic/(\im\,\xi \cap \im\,\xic)
$$ 
has trivial homology. However, 
this follows from the fact that  
$\im\,\xic$ 
and $\im\,\xi \cap \im\,\xic$
have trivial homology ($\im \, \xic$ has trivial homology
by the assumption
that $(\DDD,\pv,\phc)$ is a homological skyscraper,
and $\im\,\xi \cap \im\,\xic$ has trivial homology by Lemma~\ref{schneideritz}). 
The second claim now follows from Proposition~\ref{banalitaet}. 
\end{proof}

\begin{example}\label{exncdiff}
If $ \DDD$ is a cyclic complex  and 
$\HH(\DDD,\pv,\phc)$ is its cyclic homology 
as in Definition~\ref{heriot}, then
$\xic=0$ and $\DDD=\bar\DDD$, 
so the above lemma becomes trivial. 
However, let $A$ be an associative algebra and 
${}_ \sigma A$ be the $A$-bimodule
which is $A$ as a right $A$-module but whose
left action is given by $x \triangleright y:=\sigma (x)
y$ for some algebra endomorphism $ \sigma $.
Now consider the 
noncommutative differential forms over 
$A$ as defined in
(\ref{ncdiffforms}), but with the
boundary map $\pv$ that computes the Hochschild homology  
of $A$ with coefficients in ${}_\sigma
A$. Explicitly, $\pv$ is  
given by
\begin{align*}
 a_0 \otimes _k a_1 \otimes _k \cdots 
	\otimes _k a_n 
& \mapsto  a_0 a_1 \otimes _k a_2 \otimes _k \cdots 
	\otimes _k a_n\\
	& - a_0 \otimes _k a_1a_2 \otimes _k
\cdots 
	\otimes _k a_n + \ldots\\ 
& +(-1)^{n-1}
	a_0 \otimes _k a_1 \otimes _k \cdots 
	\otimes _k a_{n-1} a_n\\
& +(-1)^{n}
	\sigma (a_n)a_0 \otimes _k a_1 \otimes _k \cdots 
	\otimes _k a_{n-1}.
\end{align*}
With $\ph_n$ from (\ref{extrdeg}) and $c_n$ as in
Definition~\ref{heriot}, the operator $\phc$ is
induced by
$$
	\phc_n(a_0 \otimes _k \cdots 
	\otimes _k a_n)=
	\sum_{i=0}^n
	1 \otimes _k t^i (a_0 \otimes _k \cdots 
	\otimes _k a_n),
$$
where 
$$
	t(a_0 \otimes _k \cdots \otimes _k a_n) :=
	(-1)^n \sigma (a_n) \otimes _k 
	a_0 \otimes _k \cdots \otimes _k 
	a_{n-1}.
$$
In this case, $\HH(\DDD,\pv,\phc)$ is the \emph{twisted cyclic
homology} $HC^\sigma(A)$ of
$A$ that was first considered by Kustermans, Murphy and
Tuset \cite{kmt}. 
The operators $ \xi $ and $ \upsilon $ are given in
this case by
\begin{align*}
	\xi (a_0 \otimes _k \cdots \otimes _k a_n)
&=
	a_0 \otimes _k \cdots \otimes _k a_n 
	-t(a_0 \otimes _k \cdots \otimes _k a_n)	
	\\
&\quad
	+(-1)^n 1 \otimes _k 
	\sigma (a_n)a_0
	\otimes _k \cdots 
	\otimes _k a_{n-1}\\
	\upsilon (a_0 \otimes _k 
	\cdots \otimes _k a_n)
&=a_0 \otimes _k \cdots \otimes _k 
	a_n -\sigma (a_0) \otimes _k 
	\cdots \otimes _k
	\sigma (a_n).
\end{align*}
In particular, $\DDD$ is cyclic if and only if 
$ \sigma = \mathrm{id} $.

To generalise the theory of Cuntz and Quillen 
(which concerns the case where $ \sigma = \mathrm{id}
$) to this setting was one of
our original aims, motivated in particular by 
Shapiro's extension \cite{shapiro} of Karoubi's
noncommutative De Rham theory.
\end{example}
 
\subsection{The quasiisomorphism 
$ \ke \,\bar\xi^2 \rightarrow \bar\DDD$}
From now on, we will study the mixed complex 
$\bar\DDD$ in further detail. This is where the second
main assumption of Theorem~\ref{spektral} becomes
relevant, namely that the constant coefficients of the
polynomials $c_n$ are all invertible. We tacitly assume
this for the rest of the paper.

\begin{lemma}\label{directsum}
We have 
\begin{equation}\label{kersims}
	\bar\DDD=\ke\,\bar\xi^2 \oplus \im\,\bar\xi^2.
\end{equation}
\end{lemma}
\begin{proof}
That all $c_n$ are invertible in 
$k[\![x]\!]$ means their constant coefficients 
are invertible in $k$. Hence also $c_{n-1}c_n$ has 
an invertible constant coefficient $ \varepsilon_n \in
k$. Let $ \delta_n,\gamma_n \in k$ be its linear
and quadratic coefficient, 
$$
	c_{n-1}c_n=\varepsilon_n+\delta_n x + \gamma_n x^2+\ldots
$$
and define
$$
	\bar p_n:=\varepsilon_n^{-2}
	(\varepsilon_n-\delta_n \bar\xi) \bar c_{n-1}\bar c_n=
	1+\left(\frac{\gamma_n}{\varepsilon_n}-\frac{\delta_n^2}
	{\varepsilon_n^2}\right)
	\bar\xi^2 + \ldots, 
$$ 
where $\bar c_n : \bar\DDD_n \rightarrow \bar \DDD_n$ 
is the map obtained by inserting $\bar\xi$ into 
$c_n$. 

Since $\xic$ induces the trivial map on 
$\bar\DDD=\DDD/\im\,\xic$, 
Proposition~\ref{relazion} implies
\begin{equation}\label{schotten}
	\bar\xi^2\bar c_{n-1}\bar c_n=
	\bar c_{n-1}\bar c_n\bar\xi^2=0,
\end{equation}
so we get 
$$
	\im\,\bar p_n \subset \im\,\bar c_{n-1}\bar c_n \subset
	\ke\,\bar\xi^2,\quad
	\im\,\bar\xi^2 \subset
	\ke\,\bar c_{n-1}\bar c_n \subset
	\ke\,\bar p_n.
$$ 
Conversely, $\bar p_n$ 
acts by definition as the identity on $\ke\,\bar \xi^2$,
so we also have $\ke\,\bar\xi^2 \subset \im\,\bar p_n$, 
and on $\ke\,\bar p_n$ we have $1=\bar\xi^2
	\left(\frac{\delta_n^2}
	{\varepsilon_n^2}-\frac{\gamma_n}{\varepsilon_n}\right)
	+\ldots$, 
so $\ke\,\bar p_n \subset \im\,\bar\xi^2$.
It
follows that $\ke\,\bar\xi^2=\im\,\bar p_n$ and 
$\im\,\bar\xi^2=\ke\,\bar p_n$, and
also that
$\bar p_n^2=\bar p_n$,  
so $\DDD/\im\,\xic=\im\,\bar p_n \oplus \ke\,\bar p_n$. 
\end{proof}

\begin{lemma}\label{fantasie}
The inclusion $\ke\,\bar\xi^2 \rightarrow \bar\DDD$ induces 
isomorphisms 
$$
	\HH(\bar\DDD,\bar\pv,\bar\ph)	\cong 
	\HH(\ke\,\bar\xi^2,\bar\pv,\bar\ph),\quad
	\HH(\bar\DDD,\bar\pv,\bar\phc) \cong 
	\HH(\ke\,\bar\xi^2,\bar \pv,\bar\phc).
$$
\end{lemma}
\begin{proof}
As $\bar\xi$ is a morphism of mixed complexes, 
(\ref{kersims}) 
is a decomposition of mixed complexes. Since 
we have 
$$
	\ke\,\bar\xi \subset \ke\,\bar\xi^2,\quad
	\im\,\bar\xi^2 \subset
	\im\,\bar\xi,
$$
we conclude, using Lemma~\ref{directsum} in the last
step, that
\begin{equation}\label{isosinlemma4}
	\tilde\DDD=
	\DDD/(\im\,\xic+\im\,\xi) \cong 
	\bar\DDD/\im\,\bar \xi 
	\cong
	\ke\,\bar\xi^2/\im \,\bar\xi,
\end{equation}
so the
first isomorphism is obvious. Equation (\ref{kersims}) 
also implies
that $\bar\xi^2$ and hence  
$\bar\xi$ is invertible on 
$\im\,\bar\xi^2$, so $\im\,\bar\xi^2$ 
is contractible as explained 
in
Section~\ref{asasheaf}. This means the inclusion 
is a quasiisomorphism with respect
to $\bar\pv$. 
Hence the second isomorphism
follows from Proposition~\ref{banalitaet}.
\end{proof}

For later use, we record here another 
elementary consequence 
of Lemma~\ref{directsum}:
\begin{corollary}\label{tuesday}
We have 
$
	\im\,\bar\xi \cap \ke\,\bar\xi^2
	=
	\im\,\bar\xi \cap \ke\,\bar\xi.
$
\end{corollary}
\begin{proof}
Given $y=\bar \xi(x) \in \ke\,\bar\xi^2$, decompose 
$x$ as $x=v+w$ with $v \in \ke\,\bar\xi^2$ and 
$w \in \im\,\bar \xi^2$. Then 
$\bar \xi^2(y)=0$ means 
$\bar\xi^3(v)+\bar\xi^3(w)=0$; so $\bar\xi^2(v)=0$ yields 
$\bar\xi^3(w)=0$. However, $\bar\xi$ is injective on
$\im\,\xi^2$ as already remarked in the previous proof, 
so $w=0$, hence $x=v$, so
$\bar\xi (y)=\bar\xi^2(x)=\bar\xi^2(v)=0$.
\end{proof}

\begin{remark}
All the above computations are abstractions of those
made by Cuntz and Quillen for the noncommutative differential forms over an
associative algebra \cite{cq}. Informally speaking, the
message of Lemma~\ref{fantasie} can be stated as
follows:
the ``best'' mixed complexes are those where $\xi=0$, as
one can compute $\HH(\DDD,\pv,\ph)$ straight from $\DDD$ using
a spectral sequence. The second best case is   
${\DDD}=\ke \, \xi \oplus \im \, \xi$; as mentioned
after Definition~\ref{skyscrape} this means $\xi$ 
vanishes in a strong homotopical sense.
Lemma~\ref{fantasie} tells us that in general
$\xi^2$ vanishes in this homotopical sense, so $\xi$ is 
homotopically infinitesimal. 
\end{remark}

\subsection{The second exact sequence}
We now will derive the second of the two short
exact sequences in Theorem~\ref{spektral}.

First, we need the following computation:

\begin{lemma}\label{imishom}
On $\ke\,\bar\xi^2$,  
we have $\bar\pv \bar\xi=\bar\ph \bar\xi=0$, 
${\bar\phc}_n:=\bar c_n\bar \ph_n =\beta_n \bar\ph_n$
for all $n \ge 0$, where 
$\beta_n \in k$ is the constant coefficient 
of $c_n$,  and we have 
$$
	\bar\xi_n =
	(1-\frac{\beta_{n-1}}{\beta_n})
	\bar\ph_{n-1} \bar\pv_n =
	(1-\frac{\beta_n}{\beta_{n-1}})
	\bar\pv_{n+1} \bar\ph_n.
$$
\end{lemma}
\begin{proof}
Multiplying the second expression for $\xic=0$ 
in (\ref{watt}) in Proposition~\ref{relazion} on the left by 
$\pv_n$ and
using $\bar\xi^2=0$ gives
$$
	\bar \pv_n \bar\xi \bar c_{n-1}=
	\beta_{n-1} \bar \pv_n \bar\xi = 0,
$$
so $\bar\pv \bar\xi=0$ as all $ \beta_n$ are invertible.
Similarly, one obtains $\bar\ph \bar\xi=0$.  
The fact that $\bar\phc_n=\beta_n \ph_n$ is an immediate
consequence, and the formulas for $\bar\xi_n$ are obtained
by direct computation:
\begin{align*}
	\bar\xi_n  
& = \bar\ph_{n-1} \bar\pv_n  + 
	\bar\pv_{n+1} \bar\ph_n =
	\bar\ph_{n-1} \bar\pv_n +
	\beta_n^{-1} \bar \pv_{n+1} \bar {\phc}_n \\
& =
	\bar\ph_{n-1} \bar\pv_n -
	\beta_n^{-1} \bar {\phc}_{n-1} \bar \pv_n =	
	(1-\frac{\beta_{n-1}}{\beta_n})
	\bar\ph_{n-1} \bar \pv_n
\end{align*}
and similarly
$\bar\xi_n = (1-\frac{\beta_n}{\beta_{n-1}})
	\bar\pv_{n+1} \bar\ph_n$.
\end{proof}

Additionally, we will utilise the
following general statement (recall $\hat \DDD=\DDD/\im\,\xi$): 

\begin{lemma}\label{seslemma}
If $(\DDD,\pv,\ph)$ is a mixed complex with 
$\xi^2=\xic=0$, then  for all $n \ge 0$,
there are short exact sequences
$$
	\xymatrix{
	0 \ar[r] &
	\HH_n(\DDD,\pv,\phc) \ar[r] &
	\HH_n(\hat\DDD,\hat\pv,\hat\phc) \ar[r] &
	\bigoplus_{i \ge 0} 
	\im\,\xi_{n-1-2i} \ar[r] &
	0}
$$
\end{lemma}
\begin{proof}
The short exact sequence
$$
	\xymatrix{
	0 \ar[r] &
	(\im\,\xi,\pv,{\phc}) \ar[r] &
	(\DDD,\pv,{\phc}) \ar[r] &
	(\hat\DDD,\hat\pv,\hat {\phc}) \ar[r] &
	0}
$$
of mixed complexes induces short exact sequences
of the total complexes
\begin{equation}\label{manage}
	\xymatrix{
	0 \ar[r] &
	\tot(\im\,\xi) \ar[r] &
	\tot(\DDD) \ar[r] &
	\tot(\hat\DDD) \ar[r] &
	0}
\end{equation}
whose differential is $b+{\phc}$ (recall that 
$\xic=\pv {\phc}+{\phc}\pv=0$ so that $\tot_n(\DDD)=\DDD_n
\oplus \DDD_{n-2} \ldots$ here).
However, by Lemma~\ref{imishom}, $\pv+{\phc}$ vanishes
on $\im\,\xi$, so $\tot(\im\,\xi)$ is its own
homo\-lo\-gy. 
Furthermore, the inclusion $\im\,\xi \rightarrow 
\DDD$ induces the trivial map on homology, as
Lemma~\ref{imishom} implies
\begin{align*}
& \ (\xi_n x_n,\xi_{n-2} x_{n-2} ,\ldots)\\
= &\ 
	(b+{\phc})
	((1-\frac{\beta_n}{\beta_{n-1}})
	\ph_n x_n,
	(1-\frac{\beta_{n-2}}{\beta_{n-3}}) \ph_{n-2}
	x_{n-2},\ldots), 
\end{align*}
so indeed, the homology class of an element in 
$\tot(\im\,\xi)$ becomes trivial  
in $\HH(\DDD,\pv,\phc)$.
Therefore, the long exact homology sequence 
induced by (\ref{manage}) splits up into the short
exact sequences stated in the lemma.
\end{proof}

\begin{proof}[Proof of Theorem~\ref{spektral}
(\ref{part2})]
We apply Lemma~\ref{seslemma} to $\ke \, \bar\xi^2
\subset \bar\DDD$. This yields short exact sequences
\begin{equation}\label{newses}
	\xymatrix{
	0 \ar[r] &
	\HH_n(\ke\,\bar \xi^2,\bar\pv,\bar\phc) \ar[r] &
	\HH_n(\ke\,\bar\xi^2/\III,\bar\pv,\bar\phc) \ar[r] &
	\bigoplus_{i \ge 0} 
	\III_{n-1-2i}  
	\ar[r] &
	0}
\end{equation}
where we abbreviate
$$
	\III:=\im\,\bar\xi \cap \ke\,\bar\xi^2
	=
	\im\,\bar\xi \cap \ke\,\bar\xi,
$$
the second equality having been proved in
Corollary~\ref{directsum}. Note that, by abuse of notation,
we did not introduce
yet a new notation for the maps induced by
$\bar\pv,\bar\phc$ on the quotient by $\III$.   

In view of (\ref{isosinlemma4}), 
we have a commutative diagram
$$
\xymatrix{\ke\,\bar\xi^2 \ar[r] 
\ar@{^{(}->}[d] & 
\ke\,\bar\xi^2 / \III \cong \bar\DDD/\im \,\bar\xi
\ar[d]^\cong\\
\bar\DDD=\DDD/\im\,\xic \ar[r] &
\tilde\DDD=\DDD/(\im\,\xic + \im \, \xi),}
$$
where the horizontal maps are the canonical projections,
the left vertical map is the inclusion, and the right
vertical map is an isomorphism induced by this inclusion.

By Lemma~\ref{fantasie}, the left vertical arrow
induces an isomorphism 
$$
	\HH(\ke \,\bar \xi^2,\bar\pv,\bar\phc) 
	\cong \HH(\bar\DDD,\bar\pv,\bar\phc) = 
	\HH(\DDD,\pv,\phc).
$$
Similarly, the right vertical isomorphism yields an
isomorphism
$$
	\HH(\ke\,\bar\xi^2/\III,\bar\pv,\bar\phc) \cong 
	\HH(\tilde\DDD,\tilde\pv,\tilde\phc) = 
	\HH(\hat\DDD,\hat\pv,\hat\phc).
$$
These isomorphisms are compatible with the horizontal
quotient maps in the diagram. In other words, the
injectivity of the embedding 
$\HH(\ke\,\bar\xi^2,\bar\pv,\bar\phc) \rightarrow
\HH(\ke\,\bar\xi^2/\III,\bar\pv,\bar\phc)$ 
established in (\ref{newses})
transfers to injectivity of the map 
$\HH(\DDD,\pv,\phc) \rightarrow
\HH(\hat\DDD,\hat\pv,\hat\phc)$ 
induced by the quotient $\DDD
\rightarrow  \hat
\DDD$.

Since the canonical map 
$\HH(\DDD,\pv,\phc) \rightarrow
\HH(\hat\DDD,\hat\pv,\hat\phc)$ 
is injective, the long exact homology
sequence resulting from the short exact sequence
$$
	0 \rightarrow \im\,\xi \rightarrow 
	\DDD \rightarrow \hat\DDD \rightarrow 0
$$ 
splits into the short exact sequences stated in the
theorem.
\end{proof}

\begin{example}
When considering the cyclic homology of a cyclic
complex, we have $\xic=0$, hence $\bar\xi=\xi$ and 
we obtain 
$$
	\HH_n(\im\,\xi,\pv,\phc)=\bigoplus_{i \ge 0} \III_{n-2i}
	=\bigoplus_{i \ge 0} 
	\ke\,\xi_{n-2i} \cap \im\,\xi_{n-2i}.
$$
\end{example}

\subsection{The first exact sequence}\label{fesbew}
Recall that 
when $(\DDD,\pv,\phc)$ is a homological skyscraper, 
we have $\HH(\DDD,\pv,\ph) \cong 
\HH(\tilde\DDD,\tilde\pv,\tilde\ph)$ by
Lemma~\ref{wlog}, hence without loss of
generality, we work with $\tilde \DDD$ instead of 
$\DDD$ from now on.

The bulk of the remaining computations needed to prove
Theorem~\ref{spektral} are performed 
in the following lemma:

\begin{lemma}
\label{imageiso}
The map 
$
	\varphi_n : \tot (\tilde \DDD) \rightarrow \tot(\tilde
	\DDD)
$
given by
$$
	(x_n,x_{n-2},\ldots) \mapsto
	( u_n,u_{n-2},\ldots) :=
	(x_n, \beta_{n-2}^{-1} x_{n-2},
	\beta_{n-2}^{-1} \beta_{n-4}^{-1} x_{n-4},\ldots) 
$$
induces isomorphisms
\begin{align}
& \HH(\tilde \DDD/\im\,\tilde\pv,0,\tilde \ph) \cong 
	\HH(\tilde \DDD/\im\,\tilde\pv,0,\tilde
	\phc),\quad
	\HH(\im\,\tilde\pv,0,\tilde \ph) \cong
	\HH(\im \, \tilde \pv,0,\tilde\phc),\nonumber\\
\label{imageisoeqn}
&\im(\HH(\tilde\DDD,\tilde\pv,\tilde\ph) \rightarrow 
	\HH(\tilde\DDD/\im\,\tilde\pv,0,\tilde\ph)) \cong 
	\im(\HH(\tilde\DDD,\tilde\pv,\tilde\phc) \rightarrow 
	\HH(\tilde\DDD/\im\,\tilde\pv,0,\tilde\phc)),\\
\label{kernelisoeqn}	
&\ke(\HH(\tilde\DDD,\tilde\pv,\tilde\ph) \rightarrow 
	\HH(\tilde\DDD/\im\,\tilde\pv,0,\tilde\ph)) \cong 
	\ke(\HH(\tilde\DDD,\tilde\pv,\tilde\phc) \rightarrow 
	\HH(\tilde\DDD/\im\,\tilde\pv,0,\tilde\phc)).
\end{align}
\end{lemma}
\begin{proof}
Explicitly, 
a class in $\HH_n(\tilde \DDD/\im\,\tilde
\pv,0,\tilde\ph)$ 
is represented by 
an element $x=(x_n,x_{n-2},\ldots)\in \tot_n(\tilde \DDD)$ such that
there exists $y \in \tot_n(\tilde
\DDD)$ with
$$
	\tilde \pv x_{n} + \tilde \ph x_{n-2} =
	\tilde \pv y_{n},\quad
	\tilde \pv x_{n-2} + \tilde \ph x_{n-4} =
	\tilde \pv y_{n-2},\ldots
$$
The element $x$ represents the trivial homology class 
in $\HH_n(\tilde\DDD/\im\,\tilde\pv,0,\tilde\ph)$
if and
only if there are elements 
$z=(z_{n+1},z_{n-1},\ldots),t
\in \tot_{n+1}(\tilde \DDD)$ 
such that 
\begin{equation}\label{trivialclass}
	\tilde \pv z_{n+1} + \tilde \ph z_{n-1} = 
	x_{n} + \tilde \pv t_{n+1},\quad
	\tilde \pv z_{n-1} + \tilde \ph z_{n-3} = 
	x_{n-2} + \tilde \pv t_{n-1},\ldots
\end{equation}
Recall that $\tilde\xi=0$ means that 
$\tilde {\phc}_n=\beta_n \tilde\ph_n$ where 
$\beta_n \in k$ is the constant coefficient of 
$c_n$. 
Hence 
$u=\varphi_n (x) 
\in \tot_n(\tilde\DDD)$ 
satisfies 
$$
	\tilde \pv u_{n} +
	\tilde {\phc} u_{n-2} = 
	\tilde \pv v_{n},\quad
	\tilde \pv u_{n-2} + 
	\tilde {\phc} u_{n-4} =
	\tilde \pv v_{n-2},\ldots
$$
where
$v=\varphi _n(y)$.

Furthermore, (\ref{trivialclass}) implies
$$
	\tilde \pv w_{n+1} +
	\tilde {\phc} w_{n-1} = 
	u_{n} + \tilde \pv s_{n+1} ,\ldots 
$$ 
with
$
	w=\varphi_{n+1}(z),$
$s=\varphi_{n+1}(t).$
This shows that $ \varphi_n $ induces 
a well-defined map on homology which is clearly
bijective. The image of
$\HH_n(\tilde\DDD,\tilde\pv,\tilde\ph)$ in 
$\HH_n(\tilde\DDD/\im\,\tilde\pv,0,\tilde\ph)$ 
consists of those classes that can be
represented as above with
$y=0$, and then 
$v=0$ means that the
image in $\HH_n(\tilde\DDD/\im\,\tilde\pv,0,\tilde\phc)$ 
is also in
the image of $\HH_n(\tilde\DDD,\tilde\pv,\tilde\ph)$.  
The other isomorphisms follow in an exactly analogous way.
\end{proof}
\begin{remark}
For most of the isomorphisms required in Lemma~\ref{imageiso}, there is little 
restriction on the particular isomorphism we use; we could, for example, take
the identity instead of $\varphi.$ However, this causes (\ref{imageisoeqn}) to fail.
\end{remark}

\begin{proof}[Proof of Theorem~\ref{spektral}
(\ref{part1})]
The short exact sequences of chain complexes 
$$
	0 \rightarrow 
	\tot(\im\,\tilde\pv) \rightarrow 
	\tot(\tilde \DDD) \rightarrow 
	\tot(\tilde\DDD/\im\,\tilde\pv) \rightarrow 
	0  
$$
with respect to $\tilde\pv+\tilde \ph$ and 
$\tilde \pv+\tilde {\phc}$ yield long exact sequences 
$$
	\xymatrix@C=4mm{\ldots \ar[r] &
	\HH_{n+1}(\tilde\DDD/\im\,\tilde\pv,0,\tilde\ph) 
	\ar[r]^-{\partial_{n+1}^\M} &
	\HH_n(\im\,\tilde\pv,0,\tilde\ph) \ar[r] &
	\HH_n(\tilde\DDD,\tilde\pv,\tilde\ph) \ar[r] &
	\ldots}
$$
and
$$
	\xymatrix@C=4mm{\ldots \ar[r] &
	\HH_{n+1}(\tilde\DDD/\im\,\tilde\pv,0,\tilde\phc) 
	\ar[r]^-{\partial_{n+1}^\C} &
	\HH_n(\im\,\tilde\pv,0,\tilde\phc) \ar[r] &
	\HH_n(\tilde\DDD,\tilde\pv,\tilde\phc) \ar[r] &
	\ldots}
$$
These split into short exact sequences
$$
	0 \rightarrow 
	\HH_n(\im\,\tilde\pv,0,\tilde\ph) /
	\im\,\partial_{n+1}^\M \rightarrow 
	\HH_n(\tilde\DDD,\tilde\pv,\tilde\ph) \rightarrow 
	\ke\,\partial_n^\M \rightarrow 0
$$
and
$$
	0 \rightarrow 
	\HH_n(\im\,\tilde\pv,0,\tilde\phc) /
	\im\,\partial_{n+1}^\C \rightarrow 
	\HH_n(\tilde\DDD,\tilde\pv,\tilde\phc) \rightarrow 
	\ke\,\partial_n^\C \rightarrow 0.
$$
If $ \iota_n $ denotes the map
$ \HH_n(\im\,\tilde\pv,0,\tilde\phc) \rightarrow 
\HH_n(\tilde\DDD,\tilde\pv,\tilde\phc)$ induced by the
inclusion and $ \pi _n$ denotes the map 
$\HH_n (\tilde\DDD,\tilde\pv,\tilde\phc) \rightarrow 
\HH_n(\tilde\DDD/\im\,\tilde\pv,0,\tilde\phc)$,
then by exactness we have
$$
	\HH_n(\im\,\tilde\pv,0,\tilde\phc) /
	\im\,\partial_{n+1}^\C =
	\HH_n(\im\,\tilde\pv,0,\tilde\phc) /
	\ke\,\iota_n \cong
	\im\,\iota _n = 
	\ke\,\pi _n
$$
and
$$
	\ke\,\partial_n^\C =
	\im\, \pi _n \cong
	\HH_m(\tilde\DDD,\tilde\pv,\tilde\phc)/
	\ke \, \pi_n.
$$
The theorem now follows in view of the 
isomorphisms (of $k$-modules) 
$$
	\HH_n(\im\,\tilde\pv,0,\tilde\ph) /
   \im\,\partial_{n+1}^\M \cong
	\HH_n(\im\,\tilde\pv,0,\tilde\phc) /
   \im\,\partial_{n+1}^\C,\qquad
	\ke\,\partial_n^\M \cong
	\ke\,\partial_n^\C 
$$
established in Lemma~\ref{imageiso}.
Note that in the introduction we suppressed introducing 
$\tilde\DDD$, but by definition, we have
$\HH(\hat\DDD,\hat\pv,\hat \phc)=
\HH(\tilde\DDD,\tilde\pv,\tilde\phc)$.
\end{proof}

\begin{example}\label{refquest4}
If $A$ is an exterior algebra in generators $x,y$
of degree 1 
over a field $k$ of characteristic $ \neq 2$,
so that
$$
	x^2=xy+yx=y^2=0,
$$
then the noncommutative differential forms $\DDD$ over
$A$ are not a homological skyscraper. Indeed, the class
of 
$$
	1 \otimes _k xy =
	\frac{1}{2} \xi (x \otimes _k y - y \otimes _k x)  
$$
in
$\HH_1(\im\,\xi,\pv)$ is nontrivial - a straightforward
computation shows that 
$ \im\,\pv \xi $ has no component in $A_0 \otimes _k
A_2$, where $A_d$ is the degree $d$-component of $A$. 
As $\xi $ vanishes in degree 0, $1 \otimes _k xy$ also
defines a nontrivial calss in 
$\HH_1(\im\,\xi,\pv,\ph)=
\HH_1(\im\,\xi,\pv,\phc)$.
\end{example}

\begin{example}\label{example1}
For a nonstandard example of Theorem~\ref{spektral}, 
let $k$ be any commutative ring, $q \in k$, and $R$ be
the unital associative $k$-algebra generated by 
$x,y$ satisfying 
$$
	x^2=y^2=xy+qyx=0,
$$ 
so that $R$ is a free $k$-module with basis
$\{1,x,y,yx\}$.

We obtain a mixed complex of $R$-modules with
$$
	\DDD_n:=\left\{
	\begin{array}{ll}
	R/R yx \quad & n=0,\\
	R \quad	& n>0
	\end{array}
	\right.
$$
and $\pv_n$ given by right multiplication by $x$ 
and $\ph_n$ given by right multiplication by 
$y$. With $c_n=q^n$, that is, ${\phc}_n$ given by right
multiplication by $q^ny$, we obtain for $n>0$
and $r \in R=\DDD_n$
\begin{align*}
	& (\pv_{n-1} {\phc}_n +
	{\phc}_{n-1} \pv_n)r=
	r(q^nyx+q^{n-1}xy)=0,\\
	& (\pv_{n-1} \ph_n +
	\ph_{n-1} \pv_n)r=
	r(yx+xy)=(1-q)ryx,
\end{align*}
and for $n=0$
$$
	(\pv_1 {\phc}_0)(r+Ryx)=
	(\pv_1 \ph_0)(r+Ryx)=
	ryx+Ryx=
	0.
$$ 
In particular, 
$$
	\HH_1(\im\,\xi,\pv,\phc)=
	\im\,\xi_1 \cong k/I,
$$ 
where $I \lhd k$ is the
annihilator of $1-q$ in $k$. 

We furthermore see
$$
	\hat\DDD_n=\tilde\DDD_n=\left\{
	\begin{array}{ll}
	R/R yx \quad & n=0,\\
	R/R(1-q)yx \quad	& n>0,
	\end{array}
	\right.
$$
and direct computation yields 
$$
	\HH_2(\DDD,\pv,\phc) \cong k
$$
with basis given by the class of 
$(0,y + Ryx)$, while in 
$\HH_2(\hat\DDD,\hat\pv,\hat\phc)$ and 
$\HH_2(\DDD,\pv,\ph)$ 
there is an additional generator
represented by $((1-q)y,0)$,  
$$
	\HH_2(\hat\DDD,\hat\pv,\hat\phc) \cong 
	\HH_2(\DDD,\pv,\ph) \cong 
	k \oplus k/I
 \cong 
	\HH_2(\DDD,\pv,\phc) \oplus \HH_1(\im\,\xi,\pv,\phc).
$$   
Note also that $\ke\,\pi_2=0$ here, so the first
isomorphism above is canonical in this example. 
 
\end{example}

\begin{example}\label{example2}
Our final example demonstrates that 
$\ke\, \pi$ can be nontrivial. To see this, consider
the mixed complex
$$
	\DDD_n:=
	\left\{
	\begin{array}{ll}
	\mathbb{C} \quad
	& n = 0,1,2,\\
	0 \quad & n > 2,
	\end{array}\right.
$$
with (co)boundary maps
$$
	\pv_n:=
	\left\{
	\begin{array}{ll}
	\mathrm{id}  \quad
	& n = 1,\\
	0 \quad & n \neq 1,
	\end{array}\right.
	\quad
	\ph_n:=
	\left\{
	\begin{array}{ll}
	\mathrm{id}  \quad
	& n = 1,\\
	0 \quad & n \neq 1.
	\end{array}\right.
$$
Taking $c_n:=1$ for all $n$, we obtain 
$$
	\tot_n(\tilde\DDD) = 
	\left\{
	\begin{array}{ll}
	\mathbb{C} \quad & n = 0 \mbox { or } n=2k+1,\\
	\mathbb{C} \oplus \mathbb{C} \quad
	& n = 2k+2.
	\end{array}\right.
$$
Here 
$(0,1)=\tilde \pv(1) \in \ke\,\tilde \phc \cap
\im\,\tilde \pv \subset \tot_2(\tilde \DDD)$
generates $\ke\,\pi_2 \cong \mathbb{C} $, with 
$\im\,\tilde\pv \cap \im\,(\tilde \pv+\tilde \phc)=0$.
\end{example}

\end{document}